\documentclass[a4paper,10pt]{amsart}

\usepackage{amsmath,amssymb,amscd,verbatim}

\usepackage{mathptmx}                                                              % use a times font
\DeclareSymbolFont{cmletters}{OML}{cmm}{m}{it}                                     % and replace the small greek letter by the standard ones since they are too large
\DeclareSymbolFont{cmsymbols}{OMS}{cmsy}{m}{n}
\DeclareSymbolFont{cmlargesymbols}{OMX}{cmex}{m}{n}
\DeclareMathSymbol{\myjmath}{\mathord}{cmletters}{"7C}     \let\jmath\myjmath %Defining the missing commands: \jmath, \amalg and \coprod
\DeclareMathSymbol{\myamalg}{\mathbin}{cmsymbols}{"71}     \let\amalg\myamalg
\DeclareMathSymbol{\mycoprod}{\mathop}{cmlargesymbols}{"60}\let\coprod\mycoprod
\DeclareMathSymbol{\myalpha}{\mathord}{cmletters}{"0B}     \let\alpha\myalpha %Greek letters from Computer Modern
\DeclareMathSymbol{\mybeta}{\mathord}{cmletters}{"0C}      \let\beta\mybeta
\DeclareMathSymbol{\mygamma}{\mathord}{cmletters}{"0D}     \let\gamma\mygamma
\DeclareMathSymbol{\mydelta}{\mathord}{cmletters}{"0E}     \let\delta\mydelta
\DeclareMathSymbol{\myepsilon}{\mathord}{cmletters}{"0F}   \let\epsilon\myepsilon
\DeclareMathSymbol{\myzeta}{\mathord}{cmletters}{"10}      \let\zeta\myzeta
\DeclareMathSymbol{\myeta}{\mathord}{cmletters}{"11}       \let\eta\myeta
\DeclareMathSymbol{\mytheta}{\mathord}{cmletters}{"12}     \let\theta\mytheta
\DeclareMathSymbol{\myiota}{\mathord}{cmletters}{"13}      \let\iota\myiota
\DeclareMathSymbol{\mykappa}{\mathord}{cmletters}{"14}     \let\kappa\mykappa
\DeclareMathSymbol{\mylambda}{\mathord}{cmletters}{"15}    \let\lambda\mylambda
\DeclareMathSymbol{\mymu}{\mathord}{cmletters}{"16}        \let\mu\mymu
\DeclareMathSymbol{\mynu}{\mathord}{cmletters}{"17}        \let\nu\mynu
\DeclareMathSymbol{\myxi}{\mathord}{cmletters}{"18}        \let\xi\myxi
\DeclareMathSymbol{\mypi}{\mathord}{cmletters}{"19}        \let\pi\mypi
\DeclareMathSymbol{\myrho}{\mathord}{cmletters}{"1A}       \let\rho\myrho
\DeclareMathSymbol{\mysigma}{\mathord}{cmletters}{"1B}     \let\sigma\mysigma
\DeclareMathSymbol{\mytau}{\mathord}{cmletters}{"1C}       \let\tau\mytau
\DeclareMathSymbol{\myupsilon}{\mathord}{cmletters}{"1D}   \let\upsilon\myupsilon
\DeclareMathSymbol{\myphi}{\mathord}{cmletters}{"1E}       \let\phi\myphi
\DeclareMathSymbol{\mychi}{\mathord}{cmletters}{"1F}       \let\chi\mychi
\DeclareMathSymbol{\mypsi}{\mathord}{cmletters}{"20}       \let\psi\mypsi
\DeclareMathSymbol{\myomega}{\mathord}{cmletters}{"21}     \let\omega\myomega
\DeclareMathSymbol{\myvarepsilon}{\mathord}{cmletters}{"22}\let\varepsilon\myvarepsilon
\DeclareMathSymbol{\myvartheta}{\mathord}{cmletters}{"23}  \let\vartheta\myvartheta
\DeclareMathSymbol{\myvarpi}{\mathord}{cmletters}{"24}     \let\varpi\myvarpi
\DeclareMathSymbol{\myvarrho}{\mathord}{cmletters}{"25}    \let\varrho\myvarrho
\DeclareMathSymbol{\myvarsigma}{\mathord}{cmletters}{"26}  \let\varsigma\myvarsigma
\DeclareMathSymbol{\myvarphi}{\mathord}{cmletters}{"27}    \let\varphi\myvarphi

\theoremstyle{plain}
\newtheorem{thm}{Theorem}[section]
\newtheorem{cor}[thm]{Corollary}
\newtheorem{lemma}[thm]{Lemma}
\newtheorem{prop}[thm]{Proposition}
\newtheorem*{thm*}{Theorem}

\theoremstyle{definition}

\newtheorem*{df*}{Definition}

\newtheorem{rem}[thm]{Remark}
\newtheorem{ex}[thm]{Example}

\DeclareMathOperator{\Spec}{Spec}
\DeclareMathOperator{\Specmax}{Specmax}
\DeclareMathOperator{\Hom}{Hom}
\DeclareMathOperator{\colim}{colim}

\def\A{{\mathbb A}}
\def\C{{\mathbb C}}

\def\G{{\mathbb G}}
\def\P{{\mathbb P}}
\def\AA{{\textbf{A}}}
\def\cC{{\mathcal C}}
\def\cO{{\mathcal O}}

\def\fm{{\mathfrak m}}

\def\ev{\textup{ev}}
\def\id{\textup{id}}

\title{Schemes as functors on topological rings}

\author{Oliver Lorscheid}
\address{Instituto Nacional de Matem\'atica Pura e Aplicada, Estrada Dona Castorina 110, Rio de Janeiro, Brazil}
\email{lorschei@impa.br}

\author{Cec\'ilia Salgado}
\address{Instituto de Matem\'atica, Universidade Federal do Rio de Janeiro, Ilha do Fund\~ao, 21941-909, Rio de Janeiro, Brazil}
\email{salgado@im.ufrj.br}

\begin{document}

\begin{abstract}
Let $X$ be a scheme. In this text, we extend the known definitions of a topology on the set $X(R)$ of $R$-rational points from topological fields, local rings and ad\`ele rings to any ring $R$ with a topology. This definition is functorial in both $X$ and $R$, and it does not rely on any restriction on $X$ like separability or finiteness conditions. We characterize properties of $R$, such as being a topological Hausdorff ring, a local ring or having $R^\times$ as an open subset for which inversion is continuous, in terms of functorial properties of the topology of $X(R)$. Particular instances of this general approach yield a new characterization of adelic topologies, and a definition of topologies for higher local fields.
\end{abstract}

\maketitle

%%%%%%%%%%%%%%%%%%%%%%%%%%%%%%%%%%%%%%%%%%%%%%%%%%%%%%%%%%%%%%%%%%%%%%%%%%%%%%%%%%%%%%%%%%%%%%%%%%%%%%%%%%%%%%%%%%%%%%%%%%%%%%%%%%%%%%%%%%%%%%%%%%%%%%%%%%%%%%%%%%%%%%%%%%%
%%%%%%%%%%%%%%%%%%%%%%%%%%%%%%%%%%%%%%%%%%%%%%%%%%%%%%%%%%%%%%%%%%%%%%%%%%%%%%%%%%%%%%%%%%%%%%%%%%%%%%%%%%%%%%%%%%%%%%%%%%%%%%%%%%%%%%%%%%%%%%%%%%%%%%%%%%%%%%%%%%%%%%%%%%%

\section*{Introduction}
\label{intro}

\noindent
The aim of this text is to generalize the definition of a topology for the set $X(R)$ of $R$-rational points of a scheme $X$ from known cases of schemes $X$ and topological rings $R$ to all schemes and all topological rings. We will give a general definition and investigate its functorial properties. We show that it recovers the known topologies, and we characterize certain properties of $R$ in terms of functorial properties of the topologies on $X(R)$.

We begin with recalling the known constructions of topologies on sets of rational points.

\subsection*{The strong topology}
Let $k$ be a topological field with closed points. For the purpose of this text, we call a $k$-scheme of finite type a variety. It is a classical fact that one can endow the sets $X(k)$ of $k$-rational points with a topology in unique way for all varieties $X$ over $k$ such that the following properties hold for all varieties $U_i$, $X$, $Y$ and $Z$ over $k$ (cf.\ \cite[Ch.\ I.10]{Mumford90}).
\begin{enumerate}
 \item[(S0)] Every morphism $X\to Y$ yields a continuous map $X(k)\to Y(k)$.
 \item[(S1)] The canonical bijection $(X\times_Z Y)(k)\to X(k)\times_{Z(k)} Y(k)$ is a homeomorphism.
 \item[(S2)] The canonical bijection $\A^1(k)\to k$ is a homeomorphism.
 \item[(S3)] A closed immersion $Y\to X$ yields a closed embedding $Y(k)\to X(k)$ of topological spaces.
 \item[(S4)] An open immersion $Y\to X$ yields an open embedding $Y(k)\to X(k)$ of topological spaces.
 \item[(S5)] Given an affine open covering $\{U_i\}$ of $X$, a subset $W$ of $X(k)$ is open if and only if $W\cap U_i(k)$ is open in $U_i(k)$ for every $i$, and $X(k)=\bigcup  U_i(k)$.
\end{enumerate}
Since (S3) implies that this topology is stronger than the Zariski topology on $X(k)$, it is often called the \emph{strong topology for $X(k)$}.

Knowing that such a unique topology for the sets $X(k)$ exists, it is clear how to find an explicit description: if $X$ is a closed subscheme of $\A^n$, the set $X(k)$ inherits the subspace topology from $\A^n(k)=k^n$; if $X$ is a general scheme, then we can cover it with affine opens $U_i$ and endow $X(k)$ with the finest topology such that all the inclusions $U_i(k)\to X(k)$ are continuous.

The basic properties of $k$ that are used to show the independence of this construction of the ambient affine spaces $\A^n$ and of the covering $U_i$ of $X$ are the following:
\begin{enumerate}
 \item\label{item1} $k$ is a \emph{Hausdorff ring}, i.e.\ $k$ is a topological ring that is Hausdorff;
 \item\label{item2} $k$ is \emph{with open unit group}, i.e.\ the set $k^\times$ of units are open in $k$ and a topological group with the subspace topology;
 \item\label{item3} $k$ is a local ring.
\end{enumerate}

A proof of this can be found in \cite[Prop.\ 3.1]{Conrad12}. The conclusion is that the construction of the strong topology extends to all local Hausdorff rings $R$ with open unit group. The properties \eqref{item1}--\eqref{item3} essential to the construction of the strong topology for the following reasons. That $k$ has to be a Hausdorff ring can be easily derived from (S1)--(S3). That $k$ has to be with open unit group is dictated by (S4), which yields an open topological embedding $k^\times=\G_m(k)\to\A^1(k)=k$, and since the inversion of $k^\times$ is induced from the inversion of $\G_m$. The locality of $k$ is used to establish (S4) and (S5); namely, every point of $X(k)$ has to be contained in $U(k)$ for an affine open subscheme $U$ of $X$.

\subsection*{Weil's construction for ad\`ele rings}
One obstacle to extend the definition of the strong topology in terms of open coverings to other classes of topological rings is that (S4) and (S5) cannot be true anymore. If $R$ is not with open unit group, then the open immersion $\G_m\to\A^1$ yields an injection $R^\times=\G_m(R)\to\A^1(R)=R$, which is \emph{not} open. If $R$ is not local, then there are varieties $X$ that have $R$-rational points that are not contained in any affine open subscheme. This makes clear that we have to discard (S4) and (S5) if we want to find a definition of a topology for $X(R)$ for more general topological rings $R$ than local Hausdorff rings with open unit group.

An example of a ring whose units are not open is the ad\`ele ring $\AA$ of a global field $k$. For $k$-varieties $X$, Andr\'e Weil constructs in \cite{Weil82} a topology for $X(\AA)$ by a different method. We recall Weil's construction in the following; cf.\ also \cite{Oest84}.

We denote places of $k$ by $v$, the completion with respect to $v$ by $k_v$ and, in case of a non-archimedean place $v$, the ring of integers in $k_v$ by $\cO_v$. Note that all local fields $k_v$ are local Hausdorff rings with open unit group, and so are the rings $\cO_v$. Therefore the sets $X(k_v)$ come with the strong topology for every $k$-variety $X$, or, more generally, for every $k$-scheme $X$ of finite type.

Let $S$ be a finite set of places containing all the archimedean ones, and let $\cO_S$ be the $S$-integers in $k$. Let $X_S$ be an $\cO_S$-model of $X$, this is, an $\cO_S$-scheme $X_S$ such that $X_S\otimes_{\cO_S}k\simeq X$. If $v\notin S$, then also $X_S(\cO_v)=\Hom_{\cO_S}(\Spec\cO_v,X)$ is equipped with the strong topology. Note that for every finite type $k$-scheme $X$, there exist a finite set $S$ of places and a finite type $\cO_S$-model $X_S$ of $X$.

Let $\AA_S=\prod_{v\in S} k_v\times\prod_{v\notin S}\cO_v$ be the $S$-ad\`eles. We equip the set 
\[
 X_S(\AA_S) \quad = \quad \prod_{v\in S} \ X_S(k_v) \ \times \ \prod_{v\notin S} \ X_S(\cO_v)
\]
with the product topology, which we call the \emph{$S$-adelic topology}. For a finite set $S'$ of places that contains $S$, we denote by $X_{S'}$ the base extension of $X_S$ to the $S'$-integers $\cO_{S'}$. By \cite[8.14.2]{EGAIV3}, we have 
\[
 X(\AA) \quad = \quad \underset{{S\subset S'}}\colim \ \  X_{S'}(\AA_{S'}),
\]
which is equipped with the colimit topology, i.e.\ the finest topology such that all the maps $X_{S'}(\AA_{S'})\to X(\AA)$ are continuous. We call this topology in the following the \emph{adelic topology}.

\subsection*{Grothendieck's construction for affine schemes}
In the case of affine schemes $X=\Spec A$ over a base ring $k$, Grothendieck constructs a topology on $X(R)=\Hom_k(A,R)$ for any topological $k$-algebra $R$. Namely, we endow $X(R)$ with the coarsest topology such that for all elements $a\in A$, the evaluation map
\[
 \begin{array}{cccc} \ev_a: & \Hom_k(A,R) & \longrightarrow & R \\ & (f:A\to R) & \longmapsto & f(a) \end{array} 
\]
is continuous. We call this topology on $X(R)$ the \emph{affine topology}.

Note that this definition does neither require any compatibility of the topology of $R$ with the ring operation, nor a Hausdorff property nor open unit group nor a unique maximal ideal.

The affine topology coincides with the strong topology on $X(k)$ in case of a local Hausdorff ring $k$ with open unit group and a finite type $k$-scheme $X$. The affine topology also coincides with the adelic topology on $X(\AA)$ when $X$ is a variety over a global field $k$, cf.\ Corollary \ref{cor: affine=fine=strong=adelic for affine schemes}.

\subsection*{The fine topology}
Let $k$ be a ring, $X$ a $k$-scheme and $R$ a \emph{$k$-algebra with a topology}, by which we mean a $k$-algebra equipped with a topology that we do not assume to be compatible with the ring operations in any way. 

The \emph{fine topology on $X(R)$} is the finest topology such that any $k$-morphism $\varphi:U\to X$ from an affine $k$-scheme to $X$ induces a continuous map $\varphi_R:U(R)\to X(R)$ where $U(R)$ is considered with the affine topology.

\subsection*{Results}
In the subsequent sections of this text, we will prove the following statements. We show that the fine topology generalizes all the other concepts of topologies as explained above; namely, the fine topology is equal to the affine topology (Lemma \ref{lemma: fine=affine}), the strong topology (Corollary \ref{cor: fine=strong}) and the ($S$-)adelic topology (Theorem \ref{thm: fine=adelic}) whenever the latter topologies are defined for $X(R)$. 

We will see that the fine topology of $X(R)$ is functorial in $X$ and $R$ (Proposition \ref{prop: fine topology is functorial}), which can be seen as the generalization of property (S0) of the strong topology. We will investigate a series of further axioms for the fine topology, as explained in the following.

Let $\cC$ be a class of $k$-schemes. We say that \emph{$R$ satisfies \textup{(F1)--(F6)} for all schemes in $\cC$} if the following hypotheses are satisfied for all $X$, $Y$ and $Z$ in $\cC$ with respect to the fine topology:
\begin{enumerate}
 \item[(F1)] The canonical bijection $(X\times_ZY)(R)\to X(R)\times_{Z(R)}Y(R)$ is a homeomorphism.
 \item[(F2)] The canonical bijection $\A^1_k(R)\to R$ is a homeomorphism.
 \item[(F3)] A closed immersion $Y\to X$ of $k$-schemes yields a closed embedding $Y(R)\to X(R)$ of topolo\-gical spaces.
 \item[(F4)] An open immersion $Y\to X$ of $k$-schemes yields an open embedding $Y(R)\to X(R)$ of topological spaces.
 \item[(F5)] Let $\{U_i\}_{i\in I}$ be an affine open covering of $X$. Then $X(R)=\bigcup_{i\in I} U_i(R)$, and a subset $W$ of $X(R)$ is open if and only if $W\cap U_i(R)$ is open in $U_i(R)$ for all $i\in I$.
 \item[(F6)] Let $\{U_i\}_{i\in I}$ be a finite affine open covering of $X$ and $U=\coprod_{i\in I}U_i$. Let $\Psi:U\to X$ be the associated morphism. Then the map $\Psi_R:U(R)\to X(R)$ is surjective and open.
\end{enumerate}
While (F1)--(F5) are direct analogues of the properties (S1)--(S5) of the strong topology, (F6) is a variant of (F5) that remedies the fact that the adelic topology cannot be determined by affine open coverings. Namely, 
the ad\`ele ring $R=\AA$ of a global field $k$ satisfies (F6) for all $k$-schemes of finite type (Lemma \ref{lemma: U to X is adelic open}). The same holds true for the $S$-ad\`eles $\AA_S$ where $S$ is a finite set of places of $k$ (Lemma \ref{lemma: U to X is S-adelic open}).

It is clear from the previous discussions that the axioms (F1)--(F6) do not hold for all $k$-algebras $R$ with topology, but we will find necessary and sufficient conditions for various combinations of these axioms:

\begin{enumerate}
 \item $R$ is a topological ring if and only if (F1) and (F2) are satisfied for all $k$-schemes (of finite type) (Proposition \ref{prop: characterization of topological rings}).
 \item A topological ring $R$ is Hausdorff if and only if (F3) is satisfied for all $k$-schemes (of finite type) (Proposition \ref{prop: characterization of Hausdorff rings}).
 \item A local topological ring $R$ is with open unit group if and only if (F4) is satisfied for all $k$-schemes (of finite type) (Proposition \ref{prop: open unit group}).
 \item A ring is a local ring if and only if $X(R)=\bigcup_{i\in I} U_i(R)$ for all $k$-schemes $X$ and $U_i$ as in (F5) (Lemma \ref{lemma: characterization of local rings}).
% \item A local Hausdorff ring with open unit group satisfies axioms (F1)--(F5) for all $k$-schemes of finite type (Corollary \ref{cor: fine=strong}).
 \item For every two closed points in $\Spec R$, there is a decomposition $\Spec R=U_1\amalg U_2$ that separates these points if and only if $\Psi_R:U(R)\to X(R)$ is surjective for all $k$-schemes $X$ and $U$ as in (F6) (Theorem \ref{thm: characterization of rings with maximally disconnected spectrum}).
\end{enumerate}

Two particular examples for the fine topology are the following. The Zariski topology of a variety $X$ over an algebraically closed field $k$ appears naturally as the fine topology of $X(k)$ if we consider $k=\A^1(k)$ together with its Zariski topology (see Example \ref{ex: Zariski topology}). A more intriguing application are higher local fields $k$, which come equipped with a natural topology for which the multiplication fails to be continuous. This defect can be fixed by substituting the natural topology by a certain finer topology, which allows affine patchings along open immersions. However, this topology differs from the fine topology on $X(k)$ coming from the natural topology on $k$. It appears that the latter topology has not been considered yet, but it might be of use in the theory of higher local fields. For more details cf.\ Example \ref{ex: higher-dimensional local fields}.

The last section of this text is concerned with locally compact topologies. Namely, if $R$ is a locally compact Hausdorff ring that satisfies (F6) for all $k$-schemes $X$ of finite type, then $X(R)$ is locally compact in the fine topology (Lemma \ref{lemma: locally compact rings with F6}). We conclude with a question about the connection between complete varieties and compact sets of $R$-rational points.

Finally, we like to mention that besides being a generalization from certain classical cases of topologies on $X(R)$ to all $k$-schemes $X$ and all $k$-algebras $R$ with a topology, the definition of the fine topology comes in a language that transfers easily to other scheme-like theories, which includes algebraic spaces and stacks, analytic spaces and tropical schemes. In particular, we make use in a subsequent work of the very same concept to define a topology on the set of tropical points of a tropical scheme, which coincides as a point set with a tropical variety in the classical sense; cf.\ \cite{Giansiracusa13}.

\subsection*{Acknowledgements}
We are indebted to Brian Conrad who showed us a proof of the equivalence of the fine topology and the adelic topology. We would like to thank Alberto C\'amara for his explanations on higher local fields, and Gunther Cornelissen for pointing out reference \cite{Oest84}. We would like to thank an anonymous referee for pointing out mistakes in a previous version.

%%%%%%%%%%%%%%%%%%%%%%%%%%%%%%%%%%%%%%%%%%%%%%%%%%%%%%%%%%%%%%%%%%%%%%%%%%%%%%%%%%%%%%%%%%%%%%%%%%%%%%%%%%%%%%%%%%%%%%%%%%%%%%%%%%%%%%%%%%%%%%%%%%%%%%%%%%%%%%%%%%%%%%%%%%%
%%%%%%%%%%%%%%%%%%%%%%%%%%%%%%%%%%%%%%%%%%%%%%%%%%%%%%%%%%%%%%%%%%%%%%%%%%%%%%%%%%%%%%%%%%%%%%%%%%%%%%%%%%%%%%%%%%%%%%%%%%%%%%%%%%%%%%%%%%%%%%%%%%%%%%%%%%%%%%%%%%%%%%%%%%%

\section{Functoriality} 
\label{section: functoriality}

\noindent
Throughout the paper, let $k$ be a ring and $R$ a $k$-algebra with a topology, which we do not assume to be compatible with the ring operations unless stated explicitly. In this section, we will show that the fine topology for $X(R)$ is functorial in $X$ and $R$. We begin with recalling the corresponding fact for the affine topology.

\begin{lemma}\label{lemma: the affine topology is functorial}
 The affine topology of $X(R)$ is functorial in both $X$ and $R$ where $X$ is an affine $k$-scheme.
\end{lemma}

\begin{proof}
 Let $X=\Spec A$. By its definition, the affine topology of $X(R)$ is generated by the open subsets
 \[
  U_{V,a} \quad = \quad \{ \, f:A\to R \, | \, f(a)\in V\, \}
 \]
 where $a\in A$ and $V$ is an open subset of $R$. It hence suffices to show that the inverse images of subsets of the form $U_{V,a}$ are open to verify the continuity of the maps in question. 

 Let $Y=\Spec B$ and $g:B\to A$ a homomorphism of $k$-algebras that corresponds to a morphism $\varphi:X\to Y$ of affine $k$-schemes. It is easily verified that $\varphi_R^{-1}(U_{V,b})=U_{V,g(b)}$, which shows that $\varphi_R:X(R)\to Y(R)$ is continuous.

 Let $g:R\to S$ be a continuous homomorphism of $k$-algebras with topologies. It is easily verified that $g_X^{-1}(U_{V,a})=U_{g^{-1}(V),a}$, which shows that $g_X:X(R)\to X(S)$ is continuous.
\end{proof}

\begin{lemma}\label{lemma: fine=affine}
 Let $X$ be an affine $k$-scheme. Then the affine and the fine topologies on $X(R)$ coincide.
\end{lemma}

\begin{proof}
 Since $X$ is affine, the identity morphism $\id:X\to X$ yields a continuous map $\id_R:X(R)\to X(R)$ with respect to the affine topology for the domain and the fine topology for the image. This shows that the affine topology is finer than the fine topology.

 Every morphism $U\to X$ from an affine $k$-scheme $U$ to $X$ factors through the identity $\id:X\to X$, and $U(R)\to X(R)$ is continuous with respect to the affine topology for domain and image by Lemma \ref{lemma: the affine topology is functorial}. Therefore the fine topology is at least as fine as the affine topology, which completes the proof of the lemma.
\end{proof}

By virtue of this lemma, we do not have to specify the topology anymore when we are considering affine schemes. We verify the functoriality of the fine topology for $X(R)$ where $X$ can be an arbitrary $k$-scheme.

\begin{prop} \label{prop: fine topology is functorial}
 The fine topology on $X(R)$ is functorial in both $X$ and $R$.
\end{prop}

\begin{proof}
 Let $\varphi:X\to Y$ be a morphism of $k$-schemes and $\varphi_R:X(R)\to Y(R)$ the induced map. Let $W\subset Y(R)$ be open and $Z=\varphi_R^{-1}(W)$ its inverse image in $X(R)$. Consider a morphism $\alpha:U\to X$ from an affine $k$-scheme $U$ to $X$. Then the inverse image $\alpha_R^{-1}(Z)=(\varphi\circ\alpha)_R^{-1}(W)$ of $W$ is open in $U(R)$. Therefore $Z$ is open in $X(R)$, which shows that the fine topology of $X(R)$ is functorial in $X$.

 Let $f:R\to S$ be a continuous homomorphism of $k$-algebras with topologies and $f_X:X(R)\to X(S)$ the induced map. Let $W\subset X(S)$ be open and $Z=f_X^{-1}(W)$. Consider a morphism $\alpha:U\to X$ from an affine $k$-scheme $U$ to $X$. By Lemma \ref{lemma: the affine topology is functorial}, the homomorphism $f:R\to S$ induces a continuous map $f_U:U(R)\to U(S)$. Since $\alpha_S^{-1}(W)$ is open in $U(S)$, the inverse image $\alpha_R^{-1}(Z)=f_U^{-1}(\alpha_S^{-1}(W))$ is open in $U(R)$. This shows that the fine topology of $X(R)$ is functorial in $R$. 
\end{proof}

The following weakened version of (S3) survives in the generality of $k$-algebras with any topology and all $k$-schemes.

\begin{lemma}\label{lemma: closed immersions yield topological embeddings}
 Let $R$ be a $k$-algebra with topology and $\varphi:Y\to X$ a closed immersion of affine $k$-schemes. Then $\varphi_R:Y(R)\to X(R)$ is a topological embedding.
\end{lemma}

\begin{proof}
 Let $X=\Spec A$ and $Y=\Spec(A/I)$ for some ideal $I$ of $A$. The basic closed subsets of $Y(R)$ are of the form
 \[
   U_{V,\bar a} \  = \  \{ \, f: A/I \to R \, | \, f(\bar a)\in V \, \}
 \]
 where $V$ is a \emph{closed} subset of $R$ and $\bar a=a+I$ with $a\in A$. It is clear that $U_{V,\bar a}=\varphi_R^{-1}(U_{V,a})$. This proves the lemma.
\end{proof}

%%%%%%%%%%%%%%%%%%%%%%%%%%%%%%%%%%%%%%%%%%%%%%%%%%%%%%%%%%%%%%%%%%%%%%%%%%%%%%%%%%%%%%%%%%%%%%%%%%%%%%%%%%%%%%%%%%%%%%%%%%%%%%%%%%%%%%%%%%%%%%%%%%%%%%%%%%%%%%%%%%%%%%%%%%%
%%%%%%%%%%%%%%%%%%%%%%%%%%%%%%%%%%%%%%%%%%%%%%%%%%%%%%%%%%%%%%%%%%%%%%%%%%%%%%%%%%%%%%%%%%%%%%%%%%%%%%%%%%%%%%%%%%%%%%%%%%%%%%%%%%%%%%%%%%%%%%%%%%%%%%%%%%%%%%%%%%%%%%%%%%%

\section{Topological rings} 
\label{section: topological rings}

\noindent
In this section, we will show that $R$ is a topological ring if and only if it satisfies the following axioms for the class of all $k$-schemes.
\begin{enumerate}
 \item[(F1)] The canonical bijection $(X\times_ZY)(R)\to X(R)\times_{Z(R)}Y(R)$ is a homeomorphism.
 \item[(F2)] The canonical bijection $\A^1_k(R)\to R$ is a homeomorphism.
\end{enumerate}

\begin{prop} \label{prop: characterization of topological rings} 
 Given a ring $k$ and a $k$-algebra $R$ equipped with a topology. Then the following are equivalent.
 \begin{enumerate}
  \item $R$ is a topological ring.
  \item $R$ satisfies (F1) and (F2) for all schemes in $\cC$ where $\cC$ can be the class of all $k$-schemes, the class of all $k$-schemes of finite type or the class of all affine schemes of finite type over $k$.
 \end{enumerate}
\end{prop}

\begin{proof}
 If $R$ satisfies (F1) and (F2) for all $k$-schemes, then this is in particular true for all $k$-schemes of finite type. If $R$ satisfies (F1) and (F2) for all $k$-schemes of finite type, then also for all affine schemes of finite type over $k$.

 Assume that $R$ satisfies (F1) and (F2) for all affine $k$-schemes of finite type. Then $\A^1(R)=R$ by (F2) and $\A^2(R)=R\times R$ by (F1). Therefore addition and multiplication, which define morphisms $\A^2\to\A^1$ of $k$-schemes, yield continuous maps $R\times R=\A^2(R)\to \A^1(R)=R$ by the functoriality of the fine topology, cf.\ Proposition \ref{prop: fine topology is functorial}. This shows that $R$ is a topological ring.

 The proposition is proven once we have shown that a topological ring $R$ satisfies (F1) and (F2) for all $k$-schemes. We begin with the proof of (F2). The canonical bijection $\ev_T: \A^1_k(R)\to R$, which sends a homomorphism $f:k[T]\to R$ to $f(T)$, is continuous since the inverse image of an open subset $V$ of $R$ is the basis open subset
 \[
  U_{V,T} \quad = \quad \{ \, f:k[T]\to R \, | \, f(T)\in V\, \}
 \]
 of $\A^1(R)$. Conversely, consider a basis open $U_{V,p}$ of $\A^1(R)$ where $V$ is an open subset of $R$ and $p=\sum c_iT^i$ is a polynomial in $k[T]$. Since the evaluation in $p$ defines a continuous map $p:R\to R$, the inverse image $p^{-1}(V)$ is open, which equals the image of $U_{V,p}=U_{p^{-1}(V),T}$ under $\ev_T$. This shows that $\ev_T$ is a homeomorphism.
 
 The proof of (F1) for affine $k$-schemes can be found in \cite[Prop.\ 2.1]{Conrad12}. Note that the proof of this fact does not require that $X$ is of finite type. We are left with showing (F1) for the class of all $k$-schemes. By the universal property of the product of topological spaces, the canonical bijection $\Psi: (X\times_Z Y)(R)\to X(R)\times_{Z(R)} Y(R)$ is continuous. We will show that $\Psi$ is open.

 Let $X$, $Y$ and $Z$ be $k$-schemes. For an open subset $\widetilde W$ of $(X\times_Z Y)(R)$, we have to show that $W=\Psi(\widetilde W)$ is open in $X(R)\times_{Z(R)} Y(R)$. The latter topological space carries the subspace topology with respect to the canonical inclusion $\iota:X(R)\times_{Z(R)}Y(R)\to X(R)\times Y(R)$. Therefore it has a basis of open subsets of the form $\iota^{-1}(W_X\times W_Y)$ where $W_X$ is open in $X(R)$ and $W_Y$ is open in $Y(R)$. 
  
 By definition, a subset $W_X$ of $X(R)$ is open if and only if $\alpha_R^{-1}(W_X)$ is open in $U(R)$ for all morphisms $\alpha:U\to X$ where $U$ is affine, and a subset $W_Y$ of $Y(R)$ is open if and only if $\beta_R^{-1}(W_Y)$ is open in $V(R)$ for all morphisms $\beta:V\to Y$ where $V$ is affine. Thus $\iota^{-1}(W_X\times W_Y)$ is open in $X(R)\times_{Z(R)} Y(R)$ if and only if $(\alpha_R,\beta_R)^{-1}(W_X\times W_Y)$ is open in $U(R)\times V(R)$ for all morphisms $(\alpha,\beta):U\times V\to X\times_Z Y$ where $U\times V$ is affine. 

 This shows that $W$ is open in $X(R)\times_{Z(R)} Y(R)$ if and only if $(\alpha_R,\beta_R)^{-1}(W)=(\alpha,\beta)_R^{-1}(\widetilde W)$ is open in $U(R)\times V(R)=(U\times V)(R)$ for all $(\alpha,\beta):U\times V\to X\times_Z Y$ where $U\times V$ is affine. But this follows from the openness of $\widetilde W$, which finishes the proof of (F1).
\end{proof}

As a consequence of Proposition \ref{prop: characterization of topological rings} and Lemma \ref{lemma: closed immersions yield topological embeddings}, we see that the fine topology of closed subschemes of an affine space coincides with the other concepts of topologies whenever they are defined.

\begin{cor}\label{cor: affine=fine=strong=adelic for affine schemes}\ 
 \begin{enumerate}
  \item Let $R=k$ be a local Hausdorff ring with open unit group and $X$ an affine $k$-scheme of finite type. Then the fine topology coincides with the strong topology for $X(k)$.
  \item Let $\AA$ be the ad\`ele ring of a global field $k$ and $X$ an affine $k$-scheme of finite type. Then the fine topology coincides with the adelic topology for $X(\AA)$.
  \item Let $k$ be a global field, $S$ a finite set of places containing the archimedean ones and $\cO_S$ the $S$-integers. Let $\AA_S$ be the $S$-ad\`eles of $k$ and $X_S$ an affine $\cO_S$-scheme of finite type. Then the fine topology coincides with the $S$-adelic topology for $X_S(\AA_S)$.\qed
 \end{enumerate}
\end{cor}

%%%%%%%%%%%%%%%%%%%%%%%%%%%%%%%%%%%%%%%%%%%%%%%%%%%%%%%%%%%%%%%%%%%%%%%%%%%%%%%%%%%%%%%%%%%%%%%%%%%%%%%%%%%%%%%%%%%%%%%%%%%%%%%%%%%%%%%%%%%%%%%%%%%%%%%%%%%%%%%%%%%%%%%%%%%
%%%%%%%%%%%%%%%%%%%%%%%%%%%%%%%%%%%%%%%%%%%%%%%%%%%%%%%%%%%%%%%%%%%%%%%%%%%%%%%%%%%%%%%%%%%%%%%%%%%%%%%%%%%%%%%%%%%%%%%%%%%%%%%%%%%%%%%%%%%%%%%%%%%%%%%%%%%%%%%%%%%%%%%%%%%

\section{Hausdorff rings} 
\label{section: Hausdorff rings}

\noindent
Recall that we call a $k$-algebra $R$ with topology a \emph{Hausdorff ring} if is a topological ring that is Hausdorff as a topological space. In this section, we will show that a topological ring $R$ is a Hausdorff ring if and only if it satisfies the following axiom for the class of all $k$-schemes.
\begin{enumerate}\setcounter{enumi}{2}
 \item[(F3)] A closed immersion $Y\to X$ of $k$-schemes yields a closed embedding $Y(R)\to X(R)$ of topological spaces.
\end{enumerate}

\begin{prop} \label{prop: characterization of Hausdorff rings} 
 Given a ring $k$ and a $k$-algebra $R$ equipped with a topology. Then the following are equivalent.
 \begin{enumerate}
  \item $R$ is a Hausdorff ring.
  \item $R$ satisfies (F1)--(F3) for all schemes in $\cC$ where $\cC$ can be the class of all $k$-schemes, the class of all $k$-schemes of finite type or the class of all affine schemes of finite type over $k$.
 \end{enumerate}
\end{prop}

\begin{proof}
 We know already from Proposition \ref{prop: characterization of topological rings} that $R$ is a topological ring if and only if (F1) and (F2) hold true for any of the considered classes. It is clear that property (F3) specializes from a more general class of schemes to a more restrictive class. If (F3) holds true for all affine $k$-schemes of finite type, the diagonal $\Delta:\A^1_k\to\A^2_k$ induces a closed topological embedding $R=\A^1_k(R)\to \A^2_k(R)=R^2$, which shows that $R$ is Hausdorff.

 It is proven in \cite[Prop.\ 2.1]{Conrad12} that a Hausdorff ring $R$ satisfies (F3) for the class of all affine $k$-schemes. (Note that the finiteness assumption in \cite{Conrad12} is not used in this part of the proof). We show that (F3) for affine $k$-schemes implies (F3) for all $k$-schemes.

 Let $\varphi:Y\to X$ be a closed immersion of $k$-schemes and $Z\subset Y(R)$ a closed subset. Let $W=\varphi_R(Z)$ be the image in $X(R)$. We have to show that for every morphism $\alpha:U\to X$ from an affine $k$-scheme $U$ to $X$, the inverse image $\alpha_R^{-1}(W)$ is closed in $U(R)$. The pullback $\varphi^{\ast} U=U\times_X Y$ of $U$ along $\varphi$ is an affine $k$-scheme that comes with a morphism $\alpha':\varphi^*U\to Y$ and a closed immersion $\varphi':\varphi^*U\to U$. Since $Z$ is closed in $Y(R)$, the inverse image $(\alpha'_R)^{-1}(Z)$ is closed in $\varphi^*U$. By the result for closed embeddings of affine $k$-schemes, $\alpha_R^{-1}(W)=\varphi_R'((\alpha'_R)^{-1}(Z))$ is closed in $U(R)$. This shows that $W$ is closed in $X(R)$ and finishes the proof of property (F3).
\end{proof}

%%%%%%%%%%%%%%%%%%%%%%%%%%%%%%%%%%%%%%%%%%%%%%%%%%%%%%%%%%%%%%%%%%%%%%%%%%%%%%%%%%%%%%%%%%%%%%%%%%%%%%%%%%%%%%%%%%%%%%%%%%%%%%%%%%%%%%%%%%%%%%%%%%%%%%%%%%%%%%%%%%%%%%%%%%%
%%%%%%%%%%%%%%%%%%%%%%%%%%%%%%%%%%%%%%%%%%%%%%%%%%%%%%%%%%%%%%%%%%%%%%%%%%%%%%%%%%%%%%%%%%%%%%%%%%%%%%%%%%%%%%%%%%%%%%%%%%%%%%%%%%%%%%%%%%%%%%%%%%%%%%%%%%%%%%%%%%%%%%%%%%%

\section{Rings with open unit group}
\label{section: rings with open unit group}

\noindent
Recall that we say that a $k$-algebra $R$ with topology is \emph{with open unit group} if the subset $R^\times$ of units is open in $R$ and a topological group with respect to the subspace topology. Note that in the case of a topological ring $R$, the units $R^\times$ form a topological group with respect to the subspace topology if and only if the inversion $(-)^{-1}:R^\times\to R^\times$ is continuous. 

In this section, we will show that a local topological $k$-algebra $R$ is with open unit group if and only if it satisfies the following axiom for the class of all $k$-schemes.
\begin{enumerate}\setcounter{enumi}{3}
 \item[(F4)] An open subscheme $Y\to X$ of $k$-schemes yields an open embedding $Y(R)\to X(R)$ of topological spaces.
\end{enumerate}
The following weaker version of (F4) allows us to remove the locality assumption on $R$.
\begin{enumerate}\setcounter{enumi}{3}
 \item[(F4)*] A principal open subset $Y\to X$ of an affine $k$-scheme $X$ yields an open embedding $Y(R)\to X(R)$ of topological spaces.
\end{enumerate}

Recall the folling characterization of local rings.

\begin{lemma}\label{lemma: characterization of local rings}
 Given a ring $k$ and a $k$-algebra $R$ with a topology. Then the following are equivalent.
 \begin{enumerate}
  \item $R$ is a local ring.
  \item For every $k$-scheme $X$ and every open covering $\{U_i\}_{i\in I}$ of $X$, we have $X(R)=\bigcup_{i\in I} U_i(R)$.
 \end{enumerate}
\end{lemma}

\begin{proof}
 If $R$ is a local ring, $X$ a $k$-scheme and $\{U_i\}_{i\in I}$ an open covering. Every morphism $\Spec R\to X$ maps the unique closed point to a scheme-theoretic point of $X$, which is contained in one of the open subschemes $U_i$. Since generalizations are mapped to generalizations and open subsets are closed under generalizations, the image of $\Spec R\to X$ must be contained in $U_i$. This shows that \eqref{item1} implies \eqref{item2}.

 Assume that $R$ is not local. Then $R$ has two distinct maximal ideals $\fm_1$ and $\fm_2$. Let $U_1$ be the complement of $\fm_2$ and $U_2$ be the complement of $\fm_1$, which are open neighbourhoods of $\fm_1$ and $\fm_2$, respectively, that cover $X=\Spec R$. Then the $R$-rational point $\id:\Spec R\to X$ is contained in neither $U_1(R)$ nor $U_2(R)$. This shows that \eqref{item2} implies \eqref{item1}.
\end{proof}

\begin{prop}\label{prop: open unit group}
 Given a ring $k$ and a $k$-algebra $R$ with a topology. Then the following are equivalent.
 \begin{enumerate}
  \item $R$ is a topological ring with open unit group.
  \item $R$ satisfies (F1), (F2) and (F4)* for all $k$-schemes in $\cC$ where $\cC$ can be the class of all $k$-schemes, the class of all $k$-schemes of finite type or the class of all affine schemes of finite type over $k$.
 \end{enumerate}
 If, in addition, $R$ is local, then $R$ satisfies (F4).
\end{prop}

\begin{proof}
 We know already from Theorem \ref{prop: characterization of topological rings} that $R$ is a topological ring if and only if (F1) and (F2) hold true for all $k$-schemes of finite type. If \eqref{item2} holds true, then the open immersion $\G_{m,k}\to \A^1_k$ yields an open map $R^\times=\G_{m,k}(R)\to\A^1_k(R)=R$, and $R^\times=\G_{m,k}(R)$ is a topological group since $\G_{m,k}$ is a group scheme. This shows that $R$ is with open unit group.

 It is shown in \cite[Prop.\ 3.1]{Conrad12} that a local topological ring $R$ with open unit group satisfies (F4)* for the class of all affine $k$-schemes with respect to the affine topology. Note that the mentioned restriction to $k$-schemes of finite type is unnecessary by virtue of Proposition \ref{prop: characterization of topological rings}. This establishes the first claim of the proposition.
 
 We verify the latter claim of the proposition. Let $R$ be a local topological $k$-algebra with open unit group and $\iota:Y\to X$ an open immersion of $k$-schemes. Consider an open subset $Z$ of $Y(R)$ and its image $W=\iota_R(Z)$ in $X(R)$. We have to show that $\alpha_R^{-1}(W)$ is open in $U(R)$ for every morphism $\alpha:U\to X$ from an affine $k$-scheme $U$ to $X$. 
 
 The pullback $\iota^\ast U=U\times_X Y$ of $U$ along $\iota$ is a quasi-affine scheme that comes with an open immersion $\iota':\iota^\ast U\to U$ and a morphism $\alpha':\iota^\ast U\to Y$. Let $\{U_i\}$ be an open affine cover of $\iota^\ast U$ by principal open subsets $U_i$ of $U$, and let $\iota'_i:U_i\to U$ and $\alpha'_i:U_i\to Y$ be the respective restrictions of $\iota'$ and $\alpha'$ to the $U_i$. Since $R$ is local, we have $U(R)=\bigcup U_i(R)$ by Lemma \ref{lemma: characterization of local rings}. Since $Z$ is open in $Y(R)$, the inverse image $(\alpha'_i)_R^{-1}(Z)$ is open in $U_i(R)$. Since $R$ satisfies (F4)*, as shown above, $\alpha_R^{-1}(W)=\bigcup(\iota'_i)_R((\alpha'_i)_R^{-1}(Z))$ is open in $U(R)$. This verifies (F4) for $R$ and finishes the proof of the proposition.
\end{proof}

%%%%%%%%%%%%%%%%%%%%%%%%%%%%%%%%%%%%%%%%%%%%%%%%%%%%%%%%%%%%%%%%%%%%%%%%%%%%%%%%%%%%%%%%%%%%%%%%%%%%%%%%%%%%%%%%%%%%%%%%%%%%%%%%%%%%%%%%%%%%%%%%%%%%%%%%%%%%%%%%%%%%%%%%%%%
%%%%%%%%%%%%%%%%%%%%%%%%%%%%%%%%%%%%%%%%%%%%%%%%%%%%%%%%%%%%%%%%%%%%%%%%%%%%%%%%%%%%%%%%%%%%%%%%%%%%%%%%%%%%%%%%%%%%%%%%%%%%%%%%%%%%%%%%%%%%%%%%%%%%%%%%%%%%%%%%%%%%%%%%%%%

\section{Affine patchings}
\label{section: local rings}

\noindent
In order to compare the fine topology with the strong topology, we are interested in cases where the fine topology on $X(R)$ can be defined in terms of an affine open covering. More precisely, we are interested in $k$-algebras $R$ with topology that satisfy the following axiom for all $k$-schemes $X$ in a class $\cC$.
\begin{enumerate}\setcounter{enumi}{4}
 \item[(F5)] Let $\{U_i\}_{i\in I}$ be an affine open covering of $X$. Then $X(R)=\bigcup_{i\in I} U_i(R)$, and a subset $W$ of $X(R)$ is open if and only if $W\cap U_i(R)$ is open in $U_i(R)$ for all $i\in I$.
\end{enumerate}

%The latter property $X(R)=\bigcup_{i\in I} U_i(R)$ of (F5) depends only on the algebraic structure of $R$. Namely, we have the following characterization.

\begin{lemma}\label{lemma: morphisms from affine to finite type factor through an affine of finite type}
 Let $X$ be locally of finite presentation over $k$. Then every morphism $\varphi: U\to X$ from an affine $k$-scheme $U$ factors through an affine $k$-scheme $U'$ of finite type.
\end{lemma}

\begin{proof}
 Let $A$ be the $k$-algebra of global sections of $U$. The finitely generated sub-$k$-algebras $A_i$ of $A$ form an inductive system whose colimit is $A$. Consequently the spectra $U_i=\Spec A_i$, which are affine $k$-schemes of finite type, form a filtering projective system with limit $U$. By \cite[Prop.\ 8.14.2]{EGAIV3}, we have $\colim \Hom(U_i,X)=\Hom(U,X)$ if $X$ is locally of finite presentation. This shows that every morphism $U\to X$ factors through one of the $U_i$.
\end{proof}

\begin{prop}\label{prop: local and F4 imply F5}
 A local $k$-algebra $R$ with topology that satisfies (F4)* satisfies (F5) for all $k$-schemes. If $k$ is Noetherian and $R$ satisfies (F4)* for all affine $k$-schemes of finite type, then $R$ satisfies (F5) for all $k$-schemes of finite type.
\end{prop}

\begin{proof}
 Let $X$ be a $k$-scheme and $\{\iota_i:U_i\to X\}_{i\in I}$ an affine open covering of $X$. Consider a subset $W$ of $X(R)$. If $W$ is open in $X(R)$, then $\iota_{i,R}^{-1}(W)$ is open in $U_i(R)$ for all $i\in I$ by the definition of the fine topology.

 Assume, conversely, that $\iota_{i,R}^{-1}(W)$ is open in $U_i(R)$ for all $i\in I$. Consider a morphism $\alpha:U\to X$ from an arbitrary affine $k$-scheme $U$ to $X$. We have to show that $\alpha_R^{-1}(W)$ is open in $U(R)$. Since every morphism of $k$-schemes is locally affine, we find an affine open covering $\{V_i\}_{i\in I}$ of $U$ by principal open subsets such that $\alpha:U\to X$ restricts to morphisms $\alpha_i:V_i\to U_i$ of affine $k$-schemes. Thus $\alpha_{i,R}^{-1}\bigl(\iota_{i,R}^{-1}(W)\bigr)$ is open in $U(R)$ for all $i\in I$. By Lemma \ref{lemma: characterization of local rings}, every point of $W$ is contained in some $U_i(R)$, and therefore $\alpha_R^{-1}(W)=\bigcup_{i\in I}\alpha_{i,R}^{-1}\bigl(\iota_{i,R}^{-1}(W)\bigr)$, which is  an open subset of $U(R)$ since $R$ satisfies (F4)*. 

 If $k$ is Noetherian, then every $k$-scheme of finite type is of finite presentation. Therefore we can we use Lemma \ref{lemma: morphisms from affine to finite type factor through an affine of finite type} to reduce the above arguments to finite type schemes $U$ in case that $R$ satisfies (F4) only for $k$-schemes of finite type. The same proof shows that $R$ satisfies (F5) for all $k$-schemes of finite type.
\end{proof}

Since the axioms (F1)--(F5) determine the fine topology for $k$-schemes of finite type, this shows at once the equivalence of the fine topology with the strong topology in case of a local Hausdorff ring $R=k$ with open unit group.

\begin{cor}\label{cor: fine=strong}
 Let $k$ be a local Hausdorff ring with open unit group and $X$ a $k$-scheme of finite type. Then the fine topology of $X(k)$ coincides with strong topology, and $k$ satisfies (F1)--(F5) for all $k$-schemes of finite type.\qed
\end{cor}

\begin{ex}[Zariski topology]\label{ex: Zariski topology}
 Let $k$ be a field with the topology of finite closed subsets. Then the fine topology for $X(k)$ for a $k$-scheme $X$ of finite type is equal to the Zariski topology. This can be seen as follows. 

 In the affine case $X=\Spec\bigl(k[T_1,\dotsc,T_n]/I\bigr)$, a basic closed subset is of the form
 \[
  U_{\overline P,\{c\}} \ = \ \{\, f:k[T_1,\dotsc,T_n]/I \to k \, | \, f(\overline P)\in\{c\}\, \}
 \]
 where $\overline P=P+I$ is the class of a polynomial $P\in k[T_1,\dotsc,T_n]$ and $c\in k$. If $a_i=f(T_i)$, then $f(\overline P)=P(a_1,\dotsc,a_n)$. Thus $f(\overline P)\in\{c\}$ means that $P(a_1,\dotsc,a_n)-c=0$, and $U_{\overline P,c}$ corresponds with the set $V(P-c)(k)$ of $k$-rational points of the vanishing set of $P-c$. If, conversely, $V(J)$ is a closed subscheme of $X$, defined by an ideal $J$, then
 \[
  V(J)(k) \ = \ \bigcap_{\overline P\in J} \, U_{\overline P,\{0\}}.
 \]
 This shows that the fine topology is equal to the Zariski topology for affine $k$-schemes of finite type.

 Since a morphism $U\to X$ of $k$-schemes of finite type induces a continuous map $U(k)\to X(k)$ with respect to Zariski topology, the fine topology is finer than the Zariski topology. Since the Zariski topology is the finest topology such that the inclusions $U_i\to X$ for a fixed affine open covering $\{U_i\}$ of $X$ yield continuous maps $U_i(k)\to X(k)$, the fine topology for $X(k)$ is indeed equal to the Zariski topology.

 Note that if $k$ is infinite, it is neither a topological ring nor Hausdorff nor with an open unit group. The only property that remains valid from the ones that we are investigating in this text is that $k^\times$ is open in $k$. However, $k$ satisfies axioms (S2)--(S5) for all $k$-schemes of finite type; only (S1) does not hold.
\end{ex}

%\begin{comment}

\begin{ex}[The patchwork topology]\label{ex: patchwork topology}
 Let $X$ be a $k$-scheme of finite type. If $R$ is a Hausdorff ring with open unit group, i.e. $R$ satisfies (F1)--(F4) for all $k$-schemes, then we could attempt to define a \emph{patchwork topology} for $X(R)$: we say that $W\subset X(R)$ is patchwork open if and only if $W\cap U(R)$ is affine open in $U(R)$ for all affine open subschemes $U$ of $X$. Without requiring $R$ to be local, this definition entails the following problems.

 First of all, the patchwork topology cannot be defined by a single fixed covering $\{U_i\}$ of $X$ since there might be $R$-rational points that are not contained in any of the $U_i(R)$, but in $U(R)$ for some affine open $U$ of $X$ that is not among the $U_i$. This ambiguity can be resolved by considering the maximal atlas $\{U_i\}$ of all affine open subschemes of $X$.

 The more serious problem is the following. Let $k$ be a field. Then there are varieties $X$ over $k$ with pairs of $k$-rational points $\alpha$ and $\beta$ that are not contained commonly in any affine subset of $X$. Arguably the most prominent example of such a variety is Hironaka's threefold, see \cite{Hironaka62}. 

 Let $R=k\times k$ together with a topology. If, for instance, $k$ is a topological field that is Hausdorff and $R$ has the product topology, then $R$ is a Hausdorff ring with open unit group. The $R$-rational points of $X$ are pairs of $k$-rational points. By the mentioned property of $X$ with respect to the $k$-rational points $\alpha$ and $\beta$, the $R$-rational point $(\alpha,\beta)$ is not contained in $U(R)$ for any affine open subscheme $U$ of $X$. By the definition of the patchwork topology, the point $(\alpha,\beta)$ is thus an isolated point in $X(R)$.

 By Chow's lemma, there exists a birational map $\P^n\to X$ for some $n\geq 0$. Unless $R$ is discrete, $\P^n(R)$ does not have isolated components. Therefore the induced map $\P^n(R)\to X(R)$ cannot be continuous in the patchwork topology.
\end{ex}

%\end{comment}

\begin{ex}[Higher local fields]\label{ex: higher-dimensional local fields}
 A $1$-dimensional local field is, per definition, a non-archimedean local field. An $n$-dimensional local field is defined as the fraction field of a complete discrete valuation ring whose residue field is an $(n-1)$-dimensional local field. Every $n$-dimensional local field is equipped with an \emph{intrinsic topology}, which extends the topology of its residue field in a non-trivial way. In particular, it is strictly finer that the topology coming from the discrete valuation of the higher local field. See \cite{Camara11} for a definition. 

 For $n>1$, however, the multiplication fails to be continuous, and it is for this reason that the definition in terms of open affine coverings cannot be applied to higher local fields. 

 C\'amara considers in \cite{Camara11} the \emph{sequential topology on $k$}, which is a certain finer topology on $k$ than the intrinsic topology and which depends on an additional choice of a system of parameters. C\'amara defines the \emph{sequential topology on $X(k)$} for affine $k$-schemes $X$ in the usual way. Since open immersions $Y\to X$ yield open embeddings $Y(k)\to X(k)$ in the sequential topology, it is possible to extend this definition to arbitrary $k$-schemes $X$ in terms of open coverings.

 By virtue of Proposition \ref{prop: local and F4 imply F5}, the sequential topology of $X(k)$ is the same as the fine topology on $X(k)$ associated with the sequential topology on $k$.

 The sequential topology is strictly finer than the \emph{intrinsic topology for $X(k)$}, by which we mean the fine topology on $X(k)$ associated with the intrinsic topology on $k$. To our best knowledge, the intrinsic topology on $X(k)$ has not been considered before, and it might be helpful in the study of varieties over higher local fields.
\end{ex}

%%%%%%%%%%%%%%%%%%%%%%%%%%%%%%%%%%%%%%%%%%%%%%%%%%%%%%%%%%%%%%%%%%%%%%%%%%%%%%%%%%%%%%%%%%%%%%%%%%%%%%%%%%%%%%%%%%%%%%%%%%%%%%%%%%%%%%%%%%%%%%%%%%%%%%%%%%%%%%%%%%%%%%%%%%%
%%%%%%%%%%%%%%%%%%%%%%%%%%%%%%%%%%%%%%%%%%%%%%%%%%%%%%%%%%%%%%%%%%%%%%%%%%%%%%%%%%%%%%%%%%%%%%%%%%%%%%%%%%%%%%%%%%%%%%%%%%%%%%%%%%%%%%%%%%%%%%%%%%%%%%%%%%%%%%%%%%%%%%%%%%%

\section{Rings with maximally disconnected spectrum}
\label{section: rings with maximally disconnected spectrum}

\noindent
In this section, we consider a weaker version of axiom (F5) that still allows us to deduce the topology of $X(R)$ from an affine open covering of $X$, but that applies to a wider class of rings than (F5). In particular, we will show in section \ref{section: the adelic topology} that the ad\`ele ring of a global field satisfies the following axiom.

Namely, we are interested in $k$-algebras $R$ with topology that satisfy the following axiom for all $k$-schemes $X$ and $U_i$ in a class $\cC$.
\begin{enumerate}\setcounter{enumi}{5}
 \item[(F6)] Let $\{U_i\}_{i\in I}$ be a finite affine open covering of $X$ and $U=\coprod_{i\in I}U_i$. Let $\Psi:U\to X$ be the associated morphism. Then the map $\Psi_R:U(R)\to X(R)$ is surjective and open.
\end{enumerate}

Note that (F6) implies the following description of open subsets of $X(R)$ with respect to an arbitrary affine open covering of $X$.

\begin{lemma}
 Suppose that $R$ satisfies (F6) for all $k$-schemes. Let $\{\iota_i:U_i\to X\}_{i\in I}$ be an affine open covering of $X$ and denote by $\Psi_J:U_J\to X$ the induced morphism from the disjoint union $U_J=\coprod_{i\in J}U_i$ to $X$ where $J$ is a finite subset of $I$. Then a subset $W$ of $X(R)$ is open if and only if $\Psi_{J,R}^{-1}(W)$ is open in $U_J(R)$ for every finite subset $J$ of $I$.
\end{lemma}

\begin{proof}
 Clearly the inverse image of an fine open $W$ of $X(R)$ is open in any $U_J(R)$. Assume conversely that the inverse image of a subset $W$ of $X(R)$ is open in $U_J(R)$ for any finite subset $J$ of $I$. We will show that $W$ is fine open. For this purpose, let $\alpha:V\to X$ be a morphism from an affine $k$-scheme $V$ to $X$. Since $V$ is compact, the image of $\alpha$ is contained in the union $X_I=\bigcup_{i\in J} U_i$ of finitely many $U_i$. Therefore $\alpha_R^{-1}(W)=\alpha_R^{-1}(W')$ where $W'=W\cap X_I(R)$. For $W'$ and the finite covering $\{U_i\}_{i\in J}$ of $X_I$, we can apply axiom (F6) to conclude that $W'$ is fine open in $X_I(R)$ and thus $\alpha_R^{-1}(W)$ open in $V(R)$.
\end{proof}

The surjectivity of $\Psi_R:U(R)\to X(R)$ in axiom (F6) depends on a purely algebraic property of $R$. We will investigate this in the following.

We consider the maximal spectrum $\Specmax R$ of maximal ideals of $R$ as a topological subspace of $\Spec R$. We say that a ring $R$ is \emph{with maximally disconnected spectrum} if every two closed points of $\Spec R$ have respective open neighbourhoods $U_1$ and $U_2$ such that $\Spec R=U_1\amalg U_2$ as a topological space. 

Examples of rings with maximally disconnected spectrum are local rings or products of local rings, as can be deduced from condition \eqref{item3} in the following theorem. For the same reason, ad\`ele rings are with maximally disconnected spectrum. 

\begin{rem}
 The spectrum of a ring with maximally disconnected spectrum is as close to a totally disconnected space as it can be. Namely, if $x$ is a specialization of $y$, then $y$ is contained in every open neighbourhood of $x$. Therefore, the spectrum cannot be totally disconnect if it is not equal to the maximal spectrum, i.e.\ its subspace of closed points. 

 Note that the maximal spectrum of a ring with maximally disconnected spectrum is totally disconnected, but that the converse implication does not hold. To wit, a semilocal ring with connected spectrum and more than one closed point has a totally disconnected maximal spectrum, but it is not with maximally disconnected spectrum.
\end{rem}

\begin{thm}\label{thm: characterization of rings with maximally disconnected spectrum}
 The following conditions on $R$ are equivalent.
 \begin{enumerate}
  \item $R$ is with maximally disconnected spectrum.
  \item Let $X$ be a $k$-scheme with a finite affine open covering $\{U_i\}$. Let $U=\coprod U_i$ and $\Psi:U\to X$ the induced morphism. Then $\Psi_R:U(R)\to X(R)$ is surjective.
  \item For every equality in $R$ of the form $1=h_1+\dotsb+h_n$, there are idempotent elements $e_i\in (h_i)$ such that $1=e_1+\dotsb+e_n$.
 \end{enumerate}
\end{thm}

\begin{proof}
  We show that \eqref{item1} implies \eqref{item2}. Let $X$ be a $k$-scheme and $\{U_i\}$ a finite affine open covering. Let $U=\coprod U_i$ and $\Psi:U\to X$ the induced morphism. A point of $X(R)$ is a morphism $\alpha:\Spec R\to X$. We want to show that $\alpha$ factors through $U$.

 There exists a finite affine open covering $\{V_j\}$ of $\Spec R$ and a map $j\mapsto i(j)$ between the indices of the $V_j$ and the $U_i$, respectively, such that $\alpha$ restricts to morphisms $\alpha:V_j\to U_{i(j)}$ for all $j$. After relabelling indices and counting the $U_i$ multiple times if necessary, we might assume that $i=j$. We claim that we can refine the covering $\{V_i\}$ to an affine open covering $\{V_i'\}$ of $\Spec R$ such that $\Spec R=\coprod V_i'$. Once we know this, we can conclude that $\alpha:\Spec R\to X$ factors through $\coprod V_i\to \coprod U_i=U$.

 We prove the claim by induction on the number $n$ of open subsets $V_i$ of $\Spec R$ where we allow $X$ to vary. For $n=1$, there is nothing to prove. We proceed with the case $n=2$, which is the critical step in our induction. 
 
 In this case, we have affine open coverings $\Spec R=V_1\cup V_2$ and $X=U_1\cup U_2$ and restrictions $\alpha:V_i\to U_i$ for $i=1,2$. Let $Z_1$ and $Z_2$ be the respective complements of $V_1$ and $V_2$ in $\Spec R$, which are Zariski closed subsets. We have to show that there exist neighbourhoods $Z_i'$ of $Z_i$ in $\Spec R$ such that $Z=Z'_1\amalg Z'_2$. If both $Z_1$ and $Z_2$ contain a single closed point of $\Spec R$, then this follows from \eqref{item1} by the definition of a ring with maximally disconnected topological spectrum.

 Assume that $Z_1$ contains a single closed point $x$, and $Z_2$ is an arbitrary closed subset of $\Spec R$. Then there exist neighbourhoods $Z_{1,x,y}'$ of $x$ and $Z_{2,x,y}'$ of $y$ with $\Spec R=Z_{1,x,y}'\amalg Z_{2,x,y}'$ for every closed point $y\in Z_2$. Since the closed subset $Z_2$ of the compact space $\Spec R$ is compact, $Z_2$ is covered by finitely many of the $Z'_{2,x,y}$, say, by $Z'_{2,x,y_k}$ for $k=1,\dotsc,r$. Then $Z'_{1,x}=\bigcap Z'_{1,x,y_k}$ is a neighbourhood of $x$ and $Z'_{2,x}=\bigcup Z'_{2,x,y_k}$ is a neighbourhood of $Z_2$ and $\Spec R=Z'_{1,x}\amalg Z'_{2,x}$ as desired.

 We consider the general case of two closed subsets $Z_1$ and $Z_2$ of $\Spec R$. Then there are neighbourhoods $Z'_{1,x}$ of $x$ and $Z'_{2,x}$ of $Z_2$ with $Z=Z'_{1,x}\amalg Z'_{2,x}$ for every closed point $x\in Z_1$. Since $Z_1$ is compact, we can cover it with finitely many of $Z'_{1,x}$, say, with $Z'_{1,x_l}$ for $l=1,\dotsc,s$. Then  $Z'_{1}=\bigcup Z'_{1,x_l}$ is a neighbourhood of $Z_1$ and $Z'_{2}=\bigcap Z'_{2,x_l}$ is a neighbourhood of $Z_2$ and $\Spec R=Z'_{1}\cup Z'_{2}$ as desired. This completes the case $n=2$.

 Let $n>2$. Assume that the indices of the $V_i$ are $i=1,\dotsc,n$. We have proven that the covering $V_1\cup V_{\geq2}$ with $V_{\geq2}=V_2\cup\dotsb\cup V_n$ has a refinement $V'_1\subset V_1$ and $V'_{\geq2}\subset V_{\geq2}$ such that $\Spec R=V'_1\amalg V'_{\geq2}$. By the induction hypothesis, there is a refinement $\{V_i''\}$ for the covering $\{V_2,\dotsc, V_n\}$ of $V_{\geq2}$ such that $V_{\geq2}=V_2''\amalg\dotsb\amalg V_n''$. If we define $V'_i=V''_i\cap V'_{\geq2}$, then $X=V_1'\amalg\dotsb\amalg V_n'$ as desired. This shows that \eqref{item2} follows from \eqref{item1}.

 We continue with the implication \eqref{item2} to \eqref{item3}. An equality $1=h_1+\dotsc h_n$ corresponds to the covering of $X=\Spec R$ by principal opens $U_{i}=\Spec R[h_i^{-1}]$. Let $U=\coprod U_i$ and $\Psi:U\to X$ the induced morphism. The identity map $\id:\Spec R\to \Spec R$ is a point of $X(R)$. By \eqref{item2}, it factors through a morphism $\Spec R\to U$, which is only possible if there exists a refinement $\{V_i\}$ of $\{U_i\}$ such that $\Spec R=\coprod V_i$. If $R_i$ is the coordinate ring of $V_i$, then we have that $R=\prod R_i$. If we denote the image of $1$ under $R_i\to R$ by $e_i$, which is an idempotent element of $R$, then we have $V_i=\Spec R_i=\Spec R[e_i^{-1}]$. Since $X=\coprod V_i$, we have $1=e_1+\dotsb+e_n$ as desired. This shows \eqref{item3}.

 We conclude with the implication \eqref{item3} to \eqref{item1}. Consider two closed points $x$ and $y$ in $Z=\Specmax R$, which are maximal ideals of $R$. We have to find respective neighbourhoods $Z_1$ and $Z_2$ with $Z=Z_1\amalg Z_2$. To do so, we consider two elements $h_1\in x$ and $h_2\in y$ with $1=h_1+h_2$. By \eqref{item3}, there are idempotent elements $e_1\in (h_1)$ $e_2\in (h_2)$ with $1=e_1+e_2$. Since $e_1\in x$ and $e_2 \in y$, the principal open subsets $Z_i'=\Specmax R[e_i^{-1}]$ ($i=1,2$) are neighbourhoods of $x$ and $y$, respectively. Since $1=e_1+e_2$, we have $Z=Z'_1\amalg Z'_2$. This shows that $\Specmax R$ is totally disconnected and finishes the proof of the theorem.
\end{proof}

\section{The adelic topology}
\label{section: the adelic topology}

\noindent
In this section, we shall show that the fine topology coincides with the adelic and the $S$-adelic topology when $R$ is the ad\`ele ring $\AA$ of a global field $k$ or the $S$-ad\`ele ring $\AA_S$, respectively.

Let $S$ be a finite set of places of $k$ containing all the archimedean ones and let $\cO_S$ be the $S$-integers of $k$. We consider a finite type $\cO_S$-scheme $X_S$ and a finite affine open covering $\{U_{S,i}\}$. Let $U_S=\coprod U_{S,i}$ and $\Psi_S:U_S\to X_S$ the induced morphism. 

\begin{lemma}\label{lemma: U to X is S-adelic open}
 The induced map $\Psi_{S,\AA_S}:U_S(\AA_S)\to X_S(\AA_S)$ is continuous, surjective and open with respect to the $S$-adelic topology.
\end{lemma}

\begin{proof}
 By the functoriality of the $S$-adelic topology, $\Psi_{\AA}$ is continuous. By Theorem \ref{thm: characterization of rings with maximally disconnected spectrum}, $\Psi_{\AA_S}$ is surjective.

 For a place $v$ of $k$, we define $R_v$ as the completion $k_v$ of $k$ at $v$ if $v\in S$ and as the integers $\cO_v$ of $k_v$ if $v\notin S$. Since $R_v$ is a local Hausdorff ring with open unit group, we know by Corollary \ref{cor: fine=strong} that the fine topology coincides with the strong topology for $X_S(R_v)$. By Theorem \ref{thm: characterization of rings with maximally disconnected spectrum}, $\Psi_{S,R_v}:U_S(R_v)\to X_S(R_v)$ is continuous, surjective and open in the fine topology, which coincides with the strong topology in this case. Therefore the product map $\Psi_{S,\AA_S}:\prod U_S(R_v)\to \prod X_S(R_v)$ is open with respect to the product topologies, which are, by definition, the $S$-adelic topologies for $U_S(\AA_S)$ and $X(\AA_S)$, respectively.
\end{proof}

Let $X$ be a $k$-scheme of finite type and $\{U_i\}$ a finite affine open covering of $X$. Let $U=\coprod U_i$ and $\Psi:U\to X$ the induced morphism.

\begin{lemma}\label{lemma: U to X is adelic open}
 The induced map $\Psi_{\AA}:U(\AA)\to X(\AA)$ is continuous, surjective and open with respect to the adelic topology.
\end{lemma}

\begin{proof}
 By the functoriality of adelic topologies, $\Psi_\AA$ is continuous. By Theorem \ref{thm: characterization of rings with maximally disconnected spectrum}, $\Psi_\AA$ is surjective.

 Let $Z$ be an open subset of $U(\AA)$. We want to show that $W=\Psi_\AA(Z)$ is open in $X(\AA)$. Choose a finite subset $S$ of places containing all the archimedean ones and an $\cO_S$-model $\Psi_S:U_S\to X_S$ of $\Psi:U\to X$ where $U_S$ and $X_S$ are finite type $\cO_S$-models of $U$ and $X$, respectively. For finite sets $S'$ of places containing $S$, we denote by $\iota_{U,S'}:U_S(\AA_{S'})\to U(\AA)$ and $\iota_{X,S'}:X_S(\AA_{S'})\to X(\AA)$ the canonical maps induced by $\AA_{S'}\to\AA$.

 By the definition of the adelic topology, $W$ is adelic open in $U(\AA)$ if and only if $W_{S'}=\iota_{U,S'}^{-1}(W)$ is $S'$-adelic open in $U_S(\AA_{S'})$ for all finite $S'$ containing $S$. By Lemma \ref{lemma: U to X is S-adelic open}, the image $Z_{S'}=\Psi_{S,\AA_{S'}}(W_{S'})$ is $S'$-adelic open in $X_S(\AA_{S'})$ for all $S'$. Since $Z_{S'}=\iota_{X,S'}^{-1}(W)$, we conclude that $W$ is adelic open in $X(\AA)$. This completes the proof of the lemma.
\end{proof}

\begin{thm} \label{thm: fine=adelic}
 Let $k$ be a global field, $S$ a finite set of places containing all the archimedean ones and $\cO_S$ the $S$-integers.
 \begin{enumerate}
  \item Let $\AA_S$ be the $S$-ad\`eles of $k$ and $X_S$ a $\cO_S$-scheme of finite type. Then the fine topology and the $S$-adelic topology for $X_S(\AA_S)$ coincide.
  \item Let $\AA$ be the ad\`eles of $k$ and $X$ a $k$-scheme of finite type. Then the fine topology and the adelic topology for $X(\AA)$ coincide.
 \end{enumerate}
\end{thm}

\begin{proof}
 Since the proofs of \eqref{item1} and \eqref{item2} are completely analogous, we present only the proof of \eqref{item2}. We consider an adelic open subset $W$ of $X(\AA)$ and want to show that for every morphism $\alpha:U\to X$ from an affine $k$-scheme $U$ to $X$, the inverse image $Z=\alpha_\AA^{-1}(W)$ is fine open in $U(\AA)$. By Lemma \ref{lemma: morphisms from affine to finite type factor through an affine of finite type}, we can restrict ourselves to affines $U$ of finite type over $k$. Since the adelic topology is functorial in finite type $k$-schemes, $Z$ is adelic open in $U(\AA)$, and by Corollary \ref{cor: affine=fine=strong=adelic for affine schemes}, it is affine open. This shows that $W$ is fine open in $X(\AA)$.

 Assume conversely that $W$ is fine open in $X(\AA)$. Choose a finite affine open covering $\{U_i\}$ of $X$, let $U=\coprod U_i$ be the disjoint union and $\Psi:U\to X$ the induced morphism. Then $Z=\Psi_\AA^{-1}(W)$ is fine open in $U(\AA)$, and therefore adelic open by Corollary \ref{cor: affine=fine=strong=adelic for affine schemes}. By Lemma \ref{lemma: U to X is adelic open}, $W=\Psi_\AA(Z)$ is adelic open in $X(\AA)$. This finishes the proof of \eqref{item2}.
\end{proof}

\begin{cor}
 Both $\AA_S$ and $\AA$ satisfy axiom (F6) for the class of all $k$-schemes of finite type.
\end{cor}

\begin{proof}
 Axiom (F6) follows from Lemma \ref{lemma: U to X is S-adelic open} and part \eqref{item1} of Theorem \ref{thm: fine=adelic} for $\AA_S$, and from Lemma \ref{lemma: U to X is adelic open} and part \eqref{item2} of Theorem \ref{thm: fine=adelic} for $\AA$.
\end{proof}

%%%%%%%%%%%%%%%%%%%%%%%%%%%%%%%%%%%%%%%%%%%%%%%%%%%%%%%%%%%%%%%%%%%%%%%%%%%%%%%%%%%%%%%%%%%%%%%%%%%%%%%%%%%%%%%%%%%%%%%%%%%%%%%%%%%%%%%%%%%%%%%%%%%%%%%%%%%%%%%%%%%%%%%%%%%
%%%%%%%%%%%%%%%%%%%%%%%%%%%%%%%%%%%%%%%%%%%%%%%%%%%%%%%%%%%%%%%%%%%%%%%%%%%%%%%%%%%%%%%%%%%%%%%%%%%%%%%%%%%%%%%%%%%%%%%%%%%%%%%%%%%%%%%%%%%%%%%%%%%%%%%%%%%%%%%%%%%%%%%%%%%

\section{Locally compact rings}
\label{section: locally compact rings}

 \noindent
 In this concluding section, we point out that axiom (F6) is enough to ensure that a locally compact Hausdorff ring defines a functor from the category of $k$-schemes of finite type to the category of locally compact topological spaces.
  
 \begin{lemma}\label{lemma: locally compact rings with F6}
  If $R$ is a locally compact Hausdorff ring over $k$ with (F6), then $X(R)$ is locally compact for every $k$-scheme $X$ of finite type. 
 \end{lemma}

 \begin{proof}
  It is shown in \cite[Prop.\ 2.1]{Conrad12} that $X(R)$ is locally compact if $X$ is an affine $k$-scheme of finite type. For a general $k$-scheme $X$ of finite type, we choose a finite affine open covering $\{U_i\}$ and let $\Psi:U\to X$ be the associated morphism. We want to find a compact neighbourhood of a point $\alpha\in X(R)$. By (F6), there is a point $\alpha'\in U(R)$ such that $\Psi_R(\alpha')=\alpha$. By the result for affine $k$-schemes, $\alpha'$ has a compact neighbourhood $Z$, and by (F6) and the continuity of $\Psi_R$, the image $W=\Psi_R(Z)$ is a compact neighbourhood of $\alpha$ in $X(R)$.
 \end{proof}

 A natural task is to study the relation between completions of varieties and compactifications of the sets $X(R)$. It is a classical theorem that a complex variety $X$ is complete if and only if $X(\C)$ is compact in the strong topology (cf.\ \cite{Mumford90}). More generally for a local field $k$, a $k$-variety $X$ is complete if and only if for all finite field extensions $K$ of $k$, the set $X(K)$ is compact in its strong topology (cf.\ \cite{L07}). The proof in \cite{L07} can easily adopted to a proof for adelic points: a variety $X$ over a global field $k$ is complete if and only if the set $X(\AA_K)$ is compact for the ad\`eles $\AA_K$ of every finite field extension $K$ of $k$.

 We pose the following 

 \subsection*{Question}
 Let $R$ be a locally compact Hausdorff ring over $k$ with (F6). Assume that $\P^1(R)$ is compact, but that $R$ is not. Is it true that a morphism $X\to\Spec k$ is proper if and only if $X(S)$ is compact for all continuous homomorphisms $R\to S$ into a locally compact Hausdorff ring $S$ with (F6)?

%%%%%%%%%%%%%%%%%%%%%%%%%%%%%%%%%%%%%%%%%%%%%%%%%%%%%%%%%%%%%%%%%%%%%%%%%%%%%%%%%%%%%%%%%%%%%%%%%%%%%%%%%%%%%%%%%%%%%%%%%%%%%%%%%%%%%%%%%%%%%%%%%%%%%%%%%%%%%%%%%%%%%%%%%%%

\begin{small}

\bibliographystyle{plain}

\end{small}

\end{document}